\documentclass[12pt,a4paper]{amsart}
\usepackage{mathrsfs}

\newtheorem{theo+}              {Theorem}           [section]
\newtheorem{prop+}  [theo+]     {Proposition}
\newtheorem{coro+}  [theo+]     {Corollary}
\newtheorem{lemm+}  [theo+]     {Lemma}
\newtheorem{exam+}  [theo+]     {Example}
\newtheorem{rema+}  [theo+]     {Remark}
\newtheorem{defi+}  [theo+]     {Definition}

\newenvironment{theorem}{\begin{theo+}}{\end{theo+}}

\newenvironment{corollary}{\begin{coro+}}{\end{coro+}}
\newenvironment{lemma}{\begin{lemm+}}{\end{lemm+}}

\usepackage{amsthm}
\theoremstyle{plain} \theoremstyle{remark}
\newtheorem{remark}{Remark}
\newtheorem{example}{Example}

\newcommand{\grad}{\mbox{\rm grad}\,}
\def \r{\mbox{${\mathbb R}$}}
\def\E{/\kern-1.0em \equiv }

\title[Biharmonic conformal maps in dimension four]{Biharmonic conformal maps in dimension four and equations of Yamabe-type}
\author{Paul Baird}
\address{\hskip-\parindent
Laboratoire de Math\'ematiques de Bretagne Atlantique UMR 6205 \\
Universit\'e de Bretagne Occidentale, 
29238 Brest Cedex 3\\
France}
\email{Paul.Baird@univ-brest.fr}

\author{Ye-Lin Ou}

\address{\hskip-\parindent
Department of Mathematics, Texas A $\&$ M University-Commerce,
\newline Commerce,  TX 75429,  USA} 
\email{yelin.ou@tamuc.edu}
\thanks{Ye-Lin Ou was supported by grant $\#427231$ from the Simons Foundation. The author is also grateful to the Universit\'e de Bretagne Occidentale and the Laboratoire de Math\'ematiques de Bretagne Atlantique for their hospitality during a visit in May 2017 during which time most of this work was done.}

\thanks{
The authors express their thanks to J\'erome V\'etois and to Emmanuel Hebey for providing answers to questions related to Yamabe-type equations with large potentials.}
\begin{document}



\maketitle

\begin{abstract} We prove that the problem of constructing biharmonic conformal maps on a $4$-dimensional Einstein manifold reduces to a Yamabe-type equation.  This allows us to construct an infinite family of examples on the Euclidean $4$-sphere.  In addition, we characterize all solutions on Euclidean $4$-space and show that there exists at least one non-constant proper biharmonic conformal map from any closed Einstein $4$-manifold of negative Ricci curvature.
\end{abstract}

\section{Introduction}  A mapping $\phi : (M^m, g) \rightarrow (N^n, h)$ between Riemannian manifolds is called biharmonic if it is critical for the bienergy functional:
$$
\int_{M^n} |\tau_{\phi}|^2 dv_g\,,
$$
where $\tau (\phi ):= {\rm Tr}\, \nabla d \phi$ is the tension field of $\phi$. The corresponding Euler-Lagrange equations are the 4th order elliptic system:  \begin{equation} \label{bihar}
\tau_2(\phi ) := - {\rm Tr}_g (\nabla^{\phi})^2 \tau (\phi ) - {\rm Tr}_gR^N(\tau (\phi ), d\phi )d\phi = 0
\end{equation}
where $R^N$ is the Riemannian curvature on $N$.

By definition, a harmonic map has $\tau_{\phi} \equiv 0$ which is therefore automatically biharmonic; so one is interested in finding biharmonic maps which are non-harmonic, so-called \emph{proper} biharmonic maps.  One approach is to fix a map $\phi : (M^m,g) \rightarrow (N^n, h)$ between Riemannian manifolds and to deform the metric $g$ on the domain or the metric $h$ on the codomain in order to render the map biharmonic.  This idea was first considered in \cite{BK}, where it was shown that if $\widetilde{g} = e^{2\gamma}g$ is a conformally related metric ($\gamma :M^m \rightarrow \r$ a smooth function) and $\phi$ is \emph{harmonic}, then the deformed metric $\widetilde{g}$ renders $\phi$ biharmonic if and only if the gradient $\nabla\gamma$ satisfies a 2nd order partial differential equation.  Thus the 4th order system is now equivalent to two second order systems:  harmonicity of $\phi$ and the constraint on $\nabla \gamma$.   Biconformal deformations of the domain were c
 onsidered by Baird, Fardoun and Ouakkas in \cite{BFO2}.  

Conformal deformations were also considered by Ouakkas \cite{Ou-1, Ou-2} who characterized conformal biharmonic maps in terms of partial differential equations.  Equations characterizing semi-conformal bihamonic maps were obtained by Baird, Fardoun and Ouakkas \cite{BFO} and by Loubeau and Ou \cite{LO}, for which a biharmonic conformal mapping is a special case.  

To date, no example is known of a proper biharmonic submersion (even allowing for critical points) between compact manifolds $M^m \rightarrow N^n$ with $m\geq n$.  If we remove the compactness condition, then rotationally symmetric examples have been constructed between spaces of constant curvature in dimension $4$ by Montaldo, Oniciuc and Ratto \cite{MOR}.  On the other hand, many examples of immersions from compact manifolds are known, see for example \cite{BMO, CMO1, CMO2, Ou3, Ou4}. 


Examples specific to dimension $4$ show that this dimension is special for the study of biharmonic maps.  For instance, the inverse of stereographic projection $\r^4 \rightarrow S^4$ is biharmonic as well as inversion in the $3$-sphere $\r^4 \setminus \{ 0\} \rightarrow \r^4\setminus \{ 0\}$.  Neither of these mappings is biharmonic in other dimensions $n \geq 3$ and indeed, stereographic projection itself $S^4 \setminus \{ {\rm pt}\} \rightarrow \r^4$ is \emph{not} biharmonic \cite{BFO}.

 In this article, we study conformal biharmonic maps $\phi : (M^n, g) \rightarrow (N^n, h)\ (n\ge 3)$ between manifolds of the same dimension with particular emphasis on dimension $4$.  Biharmonicity of any conformal map is characterized by the requirement that its conformal factor $\lambda$ satisfies a certain 3rd order PDE \cite{BFO}.  Remarkably, on an Einstein manifold of dimension $4$, with ${\rm Ricci^M} = ag$, say, we can integrate this equation and as a consequence, we prove that biharmonic conformal maps are in correspondence with solutions to the equation
\begin{equation} \label{Ein}
\Delta \lambda - a\lambda = A\lambda^3\,,
\end{equation}
where $A$ is constant and $\lambda : M^4 \rightarrow \r$ is a smooth positive function. Furthermore, the constant $A$, the conformal factor $\lambda$ and the scalar curvature $R_h$ of $(N^4, h)$ are governed by the condition
$$
6A+\frac{2a}{\lambda^2} + R_h = 0\,.
$$
Note that here and henceforth, we take the Laplacian on functions to be $\Delta f = {\rm div}\, \grad f$, with negative spectrum.  

Equations of the type 
\begin{equation} \label{gen-Yam}
-\Delta u + k u = f u^{(n+2)/(n-2)}\,,
\end{equation}
(for appropriate functions $k$ and $f$) generalize the Yamabe equation:
\begin{equation} \label{Ya-1}
-\frac{4(n-1)}{n-2} \Delta u + R_gu = Ru^{(n+2)/(n-2)}
\end{equation}
where $R_g$ is the scalar curvature of $g$ and $R$ is constant.  As is well-known, finding a positive solution to this equation yields a metric of constant scalar curvature $R$ conformal to $g$.  When the potential $k$ in \eqref{gen-Yam} is constant, provided it is less than the potential $\frac{(n-2)}{4(n-1)}R_g$ that occurs in the Yamabe equation, then variational methods yield positive non-constant solutions \cite{Au-2}.  This allows us to deduce the existence of proper biharmonic conformal maps on closed Einstein $4$-manifolds of negative Ricci curvature (Corollary \ref{cor:hyp}).  However, when $k$ is greater than the potential in the Yamabe equation, the situation becomes more delicate.  This is the case for the biharmonic conformal map equation \eqref{Ein} on the sphere $S^4$, when we have $a = 3$.   Thanks to recent work of V\'etois and Wang \cite{VW}, we are able to show the existence of infinitely many non-constant solutions, each yielding a proper biharmonic confor
 mal mapping on the $4$-sphere (Corollary \ref{cor:sph}). 

 If $a=0$, then the biharmonic conformal map equation \eqref{Ein} becomes the standard Yamabe equation, and we give a complete characterization of the solutions on $\r^4$ (Section \ref{sec:3}).  In the final section, we characterize the biharmonicity of the M\"obius transformations in dimension $4$.  

\section{Biharmonic conformal maps} \label{sec:two}

Some examples and non-examples of conformal biharmonic maps between manifolds of the same dimension are as follows:

\medskip

\noindent \emph{Inverion in the sphere}: $\phi: \r^n\setminus\{0\}\longrightarrow \r^n\setminus \{ 0\}$ defined by $\phi(x)=\frac{x}{|x|^2}$ is biharmonic if and only if $n=4$ \cite{BK}.

\medskip

\noindent \emph {The identity map into the Poincar\'e ball}: $I: (B^n, dx^2)\rightarrow (B^n, \frac{4}{(1-|x|^2)^2}dx^2)$ (where $dx^2$ is the standard Euclidean metric) is biharmonic if and only if $n=4$ \cite{LO}.

\medskip

\noindent \emph{The identity map into the sphere}: $I : (\r^n, dx^2)\longrightarrow (S^n\setminus\{N\}, \frac{4}{(1+|x|^2)^2}dx^2)$ is biharmonic if and only if $n=4$ \cite{LO}.

\medskip

On the other hand, the inverse mappings for the second two examples above are not biharmonic in any dimension.




In general, the biharmonicity of any conformal map in dimension $n \geq 3$ is characterized by a PDE in its conformal factor.  

\begin{theorem} {\rm \cite{BFO}} A conformal map $\phi :(M^n, g) \rightarrow (N^n, h)$ with $\phi^{*}h=\lambda^2g$ and $n\ge 3$ is biharmonic if and only if
\begin{eqnarray}\label{bfo}
&&\grad( \Delta \ln\lambda) -\{ 2\Delta(\ln \lambda)  +  (n-2) |\grad \ln \lambda|^2 \}\grad \ln \lambda\\\notag
&&+ 2{\rm Ricci}^M(\grad \ln \lambda )+ \frac{6-n}{2} \grad|\grad \ln \lambda|^2 =0, 
\end{eqnarray}
\end{theorem}

Note that the fundamental equation of a semi-conformal submersion (see \cite{BW}) affirms that for a conformal submersion (mapping between manifolds of the same dimension)
$$
\tau (\phi ) = - (n-2) d\phi (\grad \ln \lambda )\,,
$$
so that any solution to \eqref{bfo} corresponds to a \emph{proper} biharmonic map if and only if the conformal factor $\lambda$ is non-constant.  Furthermore, any conformal mapping between manifolds of dimension $\geq 3$ cannot have critical points (\cite{BW}, Theorem 11.4.6), so we are interested in solutions $\lambda$ that are everywhere positive.  

As we now show, the above characterization can be integrated on an Einstein manifold to yield a particularly nice equation in dimension $4$. 
We begin with an elementary lemma.
\begin{lemma} \label{colinear}  Let $f_1, f_2 : (M^n, g) \rightarrow \r$ be two smooth functions whose gradients are everywhere colinear; thus, there exists a function  $\alpha : M^n \rightarrow \r$ such that:
$$
\grad f_1(x) = \alpha (x)\,\grad f_2(x) \quad \forall x \in M^n.
$$
Then if $x_0$ is a point where $\nabla f_1(x_0) \neq 0$, there is a neighbourhood $U$ of $x_0$ in which the function $f_1$ is  reparametrization of $f_2$.  In particular, there is a smooth function $u$ of a real variable such that $f_1 (x) = u(f_2(x))$, and furthermore $\alpha (x) = u'(f_2(x))$ for all $x \in U$.
\end{lemma} 
\begin{proof}
The level sets of each function $f_i \ (i = 1,2)$ are the integral submanifolds of the complementary distributions $(\grad f_i)^{\bot}$, which by hypothesis coincide on a neighbourhood where the gradients are non-zero.  But this means that $f_1$ must be a reparametrization of $f_2$ (and vice versa). If we set $f_1 = u(f_2)$, then necessarily $\grad f_1(x) = u'(f_2(x))\, \grad f_2(x)$.  
\end{proof} 

\begin{theorem}\label{MT3}
A smooth conformal map $\phi :(M^4, g) \rightarrow (N^4, h)$ from an Einstein $4$-manifold with ${\rm Ricci}^M = ag$ and $\phi^{*}h=\lambda^2g$ for $\lambda : M^4 \rightarrow \r \ (>0)$,  is biharmonic if and only if
\begin{equation}\label{4D}
\Delta \lambda - a\lambda = A \lambda^3
\end{equation}
for some constant $A$.  Furthermore, the constant $A$, the conformal factor $\lambda$ and the scalar curvature $R_h$ of $(N^4,h)$, are governed by the equation
\begin{equation} \label{A-cf-scal}
6A+\frac{2a}{\lambda^2} + R_h = 0\,.
\end{equation}
\end{theorem}


\begin{proof}  For comparison, we begin the proof in any dimension $n \geq 3$.  Since $(M^n, g)$ is Einstein with ${\rm Ricci}^M = ag$, we have ${\rm Ricci}^M(\grad \ln \lambda ) = a \,\grad\ln \lambda$, and \eqref{bfo} becomes:
\begin{eqnarray}\label{SF}
\grad \left(\lambda\Delta \lambda + a \lambda^2 -  \frac{n-4}{2} |\grad \lambda|^2\right) = 4(\Delta \lambda)\,\grad \lambda .
\end{eqnarray}
If the Laplacian $\Delta \lambda \equiv 0$, then in dimension $n = 4$, equation \eqref{SF} implies $a \lambda \,\grad \lambda \equiv 0$, so that, either $a = 0$ or $\lambda$ is constant (in which case $\phi$ is harmonic).  Both of these cases are taken care of by \eqref{4D}.  
Otherwise, from Lemma \ref{colinear}, there exists a function $u(s)$ such that 
\begin{equation} \label{u}
\lambda\Delta \lambda + a \lambda^2 -  \frac{n-4}{2} |\grad \lambda|^2 = 4u(\lambda )
\end{equation}
and furthermore we have $\Delta \lambda = u'(\lambda )$.  Substituting this last equality back into \eqref{u} gives
\begin{equation} \label{nnot4}
\lambda u'(\lambda ) + a \lambda^2 -  \frac{n-4}{2} |\grad \lambda|^2 = 4u(\lambda )
\end{equation}
Now restrict to the case when $n = 4$.  Then a necessary and sufficient condition is that $u(\lambda )$ solves the equation
\begin{equation}
\lambda u' - 4u = - a\lambda^2
\end{equation}
whose most general solution is given by $u(\lambda ) = \tfrac{1}{2}a\lambda^2 + \tfrac{A}{4}\lambda^4$, for an arbitrary constant $A$.  This now yields $\Delta \lambda = u'(\lambda ) = a\lambda + A \lambda^3$ giving \eqref{4D}.  Conversely, if \eqref{4D} is satisfied, then so is \eqref{SF} (with $n=4$).   

In order to establish the relation \eqref{A-cf-scal} between the constant $A$, the scalar curvature $R_g$ on $(N^4, h)$ and the conformal factor $\lambda$, we note that for any conformal local diffeomorphism between manifolds of the same dimension $n \geq 2$, we have
$$
2(n-1)\Delta \lambda = \lambda R_g - \lambda^3R_h - \frac{(n-1)(n-4)}{\lambda} | \grad \lambda |^2,
$$
where $R_g$ is the scalar curvature of $(M^n,g)$ (see \cite{BW}, Proposition 11.4.2).  In the case when $n=4$ and $(M^4, g)$ is Einstein, this becomes:
$$
6\Delta \lambda = 4a \lambda - R_h \lambda^3\,.
$$
For this equation to be compatible with \eqref{4D}, we require
$$
4a \lambda - R_h \lambda^3 = 6a\lambda + 6A\lambda^3
$$
which is equivalent to the condition \eqref{A-cf-scal}.
\end{proof}

\begin{remark}  Note that the above theorem imposes no a priori condition on the constant $A$:  whatever the constant, any smooth solution to \eqref{4D} yields a conformal biharmonic map.  For example, we can take $\phi$ to be the identity and $h = \lambda^2g$.  The condition \eqref{A-cf-scal} then determines the scalar curvature $R_h$ of the codomain $(N^4, h)$ in terms of the conformal factor $\lambda$. On the other hand, if we first impose a condition on $(N^4, h)$ that restricts its scalar curvature, then this also imposes a constraint on $A$.  For example, if $R_h$ is constant and $a \neq 0$, then necessarily $\lambda$ is also constant and $\phi$ is a homothety and so harmonic.   We will discuss further consequences of this condition in what follows.
\end{remark}

\begin{remark}  The above proof shows that for $n\neq 4$, a conformal map $\phi :(M^n, g) \rightarrow (N^n, h)$ from an Einstein manifold is biharmonic if and only if $\lambda$ is isoparametric.  That is, both $\Delta \lambda$ and $|\grad \lambda |^2$ are functions of $\lambda$.  This fact was also noticed in \cite{BFO}.  Indeed, one can be more specfic from the above proof, by observing that there exists a function $u(s)$ of a real variable such that 
\begin{equation}
\Delta \lambda = u'(\lambda )  \quad {\rm and} \quad |\grad \lambda |^2 = \frac{2}{n-4} \left( \lambda u'(\lambda ) - 4 u (\lambda ) + a\lambda^2 \right)
\end{equation}
Then \eqref{SF} leads to an ordinary differential equation in $u$ which can always be solved locally.  However, no globally defined non-constant solution has been found on a Euclidean sphere $S^n$ for $n \neq 4$.  This is in contrast to what we establish below, namely that there are many non-constant solutions to the biharmonic conformal map equation \eqref{4D} on the $4$-sphere.  
\end{remark}

\begin{remark} As the following arguments show, there is no obvious way to deal with the case when the domain $(M^4, g)$ is no longer Einstein.  Then equation \eqref{SF} becomes
\begin{equation} \label{SF-bis}
\grad (\lambda \Delta \lambda ) - 4 \Delta \lambda \,\grad \lambda + 2 \lambda\, {\rm Ricci}^M(\grad \lambda ) = 0\,.
\end{equation}
Let us make the hypothesis that the conformal factor satisfies an equation of the type \eqref{4D}, but with coefficients no longer necessarily constant:
$$
\Delta \lambda = a(x)\lambda + A(x) \lambda^3\,.
$$
Substitution into \eqref{SF-bis}, now yields the requirement that
$$
2 ({\rm Ricci}^M - ag)(\grad \lambda )+ \lambda \,\grad a + \lambda^3\, \grad A = 0\,.
$$
Clearly the case ${\rm Ricci}^M = ag$ (i.e. $M^4$ Einstein) with both $a$ and $A$ constant satisfies this constraint, but otherwise, there seems to be no way to obtain a condition on $a$ and $A$ that depends only on the curvature and not on $\lambda$. 
\end{remark}

\begin{example}  The function $\lambda = \frac{1}{|x|}$ defined on $\r^4 \setminus \{ 0\}$ satisfies $\Delta \lambda = - \frac{1}{|x|^3} = - \lambda^3$ and so solves equation \eqref{4D} with $C=0$ and $A = -1$.  The corresponding conformal transformation can be realized as follows.  Use polar coordinates $(r, \theta )$ on $\r^4$, where $r = |x|$ is the radial coordinate and $\theta \in S^3$.  Thus each $x \in \r^4\setminus \{ 0\}$ can be written uniquely as $x = r \theta$ for $r \in (0, \infty )$ and $\theta \in S^3$.  Writing $g_{S^3}$ for the standard metric on $S^3$, the metric on $\r^4$ has the form
$$
g = dr^2 + r^2 g_{S^3} = r^2 \big( d(\ln r)^2 + g_{S^3}\big)
$$
However, the metric $h = dt^2 + g_{S^3}$ is the standard metric on the cylinder $\r \times S^3$, so we see that the mapping $\phi : \r^4 \setminus \{ 0\} \rightarrow (\r \times S^3, h)$ given by $\phi (r\theta ) = (\ln r , \theta )$ is a biharmonic conformal diffeomorphism with conformal factor $\lambda = \frac{1}{r} = \frac{1}{|x|}$. 
\end{example}

\section{The Sobolev embedding theorem and the Yamabe equation on Euclidean space} \label{sec:3}

Recall that the classical Sobolev embedding theorem says that  we can embed $W_0^{1,2}(\r^n)$ into $L^p(\r^n)$. More precisely, for any $v\in C_0^{\infty}(\r^n)$ we have 

\begin{equation}\label{YF}
c\left(\int_{\r^n}|v|^pdx   \right)^{2/p}\le \int_{\r^n}|\grad v|^2dx.
\end{equation}

It is known (see e.g., \cite{Ch}) that the best constant $c$ and the extremal functions $v$ which satisfy the inequality in (\ref{YF}) can be determined. 

\begin{theorem} {\rm (\cite{Bl}, \cite{Ta}, \cite{Au})} The best constant in the Sobolev inequality in {\rm (\ref{YF})} is $c=\frac{n(n-2)}{4}w_n^{2/n}$ and it is only realized by the functions
\begin{equation}\label{AS}
v_{\delta, x_0}(x)=\left( \frac{2\delta}{\delta^2+|x-x_0|^2}\right)^{\frac{n-2}{2}},
\end{equation}
where $(x_0, \delta)\in \r^n\times (0,\infty)$.
\end{theorem}

It is also known (see e.g., \cite{Ch}) that the Euler-Lagrange equation for the extremal functions saturating the inequality (\ref{YF}) is
\begin{equation}\label{EL}
\Delta v=-\frac{n(n-2)}{4}v^{\frac{n+2}{n-2}},\hskip1cm {\rm on}\;\; \r^n.
\end{equation}
Thus all functions $v_{\delta, x_0} $ in (\ref{AS}) are solutions of the equation (\ref{EL}).  By a theorem of Caffarelli-Gidas-Spruck, any positive solution of (\ref{EL}) is one of the $v_{\delta, x_0} $ in (\ref{AS}) \cite{CGS}.

When $n=4$, we have the following corollary which will be used in our classification of all biharmonic conformal maps from $4$-dimensional Euclidean space.
\begin{corollary} \label{cor:AS4}
Any positive solution of the equation
\begin{equation}\label{EL4}
\Delta v=-2v^3,\hskip1cm {\rm on}\;\; \r^4
\end{equation}
 is one of the functions in the family
 \begin{equation}\label{AS4}
v_{\delta , x_0}(x)= \frac{2\delta}{\delta^2+|x-x_0|^2},
\end{equation}
where $(x_0, \delta )\in \r^4\times (0,\infty)$.
\end{corollary}

In Section \S\ref{sec:mobius}, we will explore equation \eqref{4D} in respect of the full M\"obius group, both with respect to the Euclidean and spherical metrics.   Inversion provides a particular example.   
\begin{example}
Inversion in the $3$-sphere $\phi:\r^4\setminus\{0\}\longrightarrow\r^4$ with $\phi(x)=\frac{x}{|x|^2}$ is a conformal map with conformal factor $\lambda=\frac{1}{|x|^2}$.  This is a harmonic function so, by Theorem \ref{MT3}, $\phi$ is a biharmonic conformal map. This confirms a result in \cite{BK}. For the proof that this inversion is also a biharmonic morphism (a map between Riemannian manifolds that preserves the solutions of bi-Laplace equation), see \cite{LO}.
\end{example}
\begin{example}  If we return to the examples cited at the beginning of \S\ref{sec:two}, namely the identity maps into the Poincar\'e ball and the sphere, then the conformal factor is given by $\lambda=2/(1+\varepsilon|x|^2)$ with $\varepsilon = -1$ and $+1$, respectively.  A short calculation shows that $\Delta \lambda=-16\varepsilon/(1+\varepsilon|x|^2)^3$, so that $\lambda$ solves equation \eqref{4D} with $A=-2\varepsilon$.  Equation \eqref{A-cf-scal} then confirms the scalar curvature of the codomain to be $12\varepsilon$.  
\end{example}

Recall (see, e.g., \cite{CN}) that  a Riemannian metric ${\widetilde g}$ on a Riemannian manifold $(M^m, g)$ is said to be harmonic (respectively, biharmonic) with respect to $g$, if the identity map $(M^m, g)\rightarrow (M^m, {\widetilde g})$ is harmonic (respectively, biharmonic).   We will say that a metric is proper biharmonic if it is biharmonic but not harmonic.  
On the other hand, the Yamabe problem is to find  metrics of constant scalar curvature in the conformal class of $g$: $ [g] =\{ fg| \, f: M\longrightarrow (0, \infty)\}$. As is well-known, for $n \geq 3$, if one sets ${\widetilde g}=u^{\frac{4}{n-2}}g$ for some positive function $u$, then ${\widetilde g}$ has constant scalar curvature $R\in \r$ if and only if
the function $u$ satisfies the Yamabe equation \eqref{Ya-1}.  

\begin{corollary}
A conformally flat metric ${\widetilde g}=\lambda^2 dx^2$ on $U\subseteq \r^4$ is biharmonic if and only if the conformal factor is a solution of the Yamabe equation or, equivalently, the Riemannian manifold $(U, {\widetilde g}=\lambda^2 dx^2)$ has constant scalar curvature $R$ and $\Delta \lambda = -\frac{R}{6} \lambda^3$.
\end{corollary}
\begin{proof}
By definition, the conformally flat metric ${\widetilde g}=\lambda^2 ds^2$ on $U\subseteq \r^4$ is biharmonic if and only if the identity map $(\r^4\supseteq U, dx^2)\rightarrow (U\subseteq \r^4, {\widetilde g}=\lambda^2dx^2)$ is biharmonic. By Theorem \ref{MT3}, this is equivalent to $\lambda$ being a solution to the equation $\Delta \lambda= A\lambda^3$ for some constant $A$ satisfying \eqref{A-cf-scal}. But on $4$-dimensional Euclidean space $(U\subseteq \r^4, dx^2)$,  the Yamabe Equation  (\ref{Ya-1}) reduces to  $\Delta_gu =-\frac{R}{6}u^3$, where $R$ is the constant scalar curvature of the metric ${\widetilde g}=u^2g$.  The corollary now follows.
\end{proof}
\begin{example}
The Yang-Mills equation in the $4$-dimensional conformally flat space \\$(\r^4\setminus \{0\}, |x|^{-2}dx^2)$ was studied in \cite{Gu}.  We can check that this conformally flat metric is a biharmonic metric. In fact, the  identity map $I: \r^4\setminus \{0\}\rightarrow (\r^4\setminus \{0\},|x|^{2\alpha}dx^2)$ is a conformal map with conformal factor $\lambda=|x|^{\alpha}$. A straightforward computation yields $\Delta \lambda=\alpha(\alpha+2)|x|^{\alpha-2}$ and hence the equation $\Delta \lambda= A\lambda^3$ has the only solution $\alpha =-1$ with $A=-1$. In particular, the identity map $I: \r^4\setminus \{0\}\rightarrow (\r^4\setminus \{0\},|x|^{2\alpha}dx^2)$ is a biharmonic map if and only if $\alpha =-1$.  This is equivalent to the affirmation that  the conformally flat metric ${\bar g}= |x|^{2\alpha}dx^2$ has constant scalar curvature if and only if $\alpha =-1$.  In fact, equation \eqref{A-cf-scal} shows that the scalar curvature is $-6$.  
\end{example}

\section{Biharmonic maps and metrics on compact Einstein manifolds}  We turn to the case when $(M^4, g)$ is a closed manifold (i.e. a compact manifold without boundary) which is Einstein.   Set ${\rm Ricci}^M = ag$. We refer the reader to the paper of LeBrun \cite{Le} for a good account of the numerous examples of Einstein manifolds for different signs of the constant $a$.  In this context, it should be noted that to date, no example was known of a proper biharmonic submersive map from a closed manifold.  A theorem of Jiang asserts that any biharmonic map from a closed manifold into one of non-positive sectional curvature must be harmonic \cite{Ji1, Ji2}.  We recall equation \eqref{4D}: 
\begin{equation} \label{cc}
\Delta \lambda - a\lambda = A \lambda^3\,.
\end{equation}
When the constant $A$ is strictly negative, the equation may be considered as belonging to the class of equations on a closed Riemannian manifold $(M^n, g)$:
\begin{equation} \label{gen-Yam-2}
-\Delta u + k(x) u =  u^{(n+2)/(n-2)}\,,
\end{equation}
with $k$ a smooth function.  A theorem of Aubin with $n \geq 4$ states that if
\begin{equation} \label{Yam-in}
k(x_0) <  \frac{n-2}{4(n-1)}R_g(x_0),
\end{equation}
at some point $x_0$, then \eqref{gen-Yam-2} has a smooth strictly positive solution \cite{Au-1} -- see the survey of Hebey for an up-to-date account \cite{He} (\S 2.5 and \S 2.6, in particular Theorem 2.10) --  note that contrary to the citations, we take the Laplacian on functions to be $+{\rm div}\, \grad$ with negative spectrum.  On setting $A=-1$ in \eqref{cc}, since $R_g = 4a$, the inequality \eqref{Yam-in} becomes:
\begin{equation} \label{in}
a< \frac{2}{3}a
\end{equation}
which is satisfied if and only if $a<0$.  

\medskip

\noindent \emph{The case of negative Einstein constant.}  We take $A=-1$ in \eqref{cc}.  Since $a<0$, equality \eqref{in} is satisfied; furthermore, on account of the sign of $k$, the positive solution to \eqref{gen-Yam-2} cannot be constant.  Theorem \ref{MT3} then gives the following consequence.  

\begin{corollary} \label{cor:hyp}  Let $(M^4, g)$ be a closed Einstein $4$-manifold with negative Ricci curvature.  Then there is a conformally related metric $h = \lambda^2g$ which is proper biharmonic, i.e. the identity map $I: (M^4, g)\rightarrow (M^4, h=\lambda^2g)$ is a proper biharmonic map.  Furthermore, the scalar curvature of $\lambda^2 g$ is given by $R_h = 6 + \frac{6}{\lambda^2}$.
\end{corollary}

Note that this proper conformal biharmonic map transforms a metric of negative scalar curvature into one of strictly positive non-constant scalar curvature.

\medskip  

\noindent \emph{The case of Ricci flat manifolds.}  In this case we take $a=0$.  Since on a compact manifold $(M^4,g)$ without boundary we have $\int_{M^4}\Delta \lambda dv_g = 0$ and since $\lambda >0$, by Theorem \ref{MT3}, we must have $A=0$.  But then by the maximum principle, $\lambda$ is constant.  We therefore have the following non-existence result.

\begin{corollary}  Let $(M^4, g)$ be a closed Einstein $4$-manifold with vanishing Ricci curvature.  Then there is no conformal proper biharmonic map $\phi : (M^4, g) \rightarrow (N^4, h)$ into any Riemannian $4$-manifold.
\end{corollary}

\medskip

\noindent \emph{The case of positive Einstein constant.}  When $a>0$, the inequality \eqref{Yam-in} is no longer satisfied.  This is a much more delicate case and has been the subject of  intense investigation recently, in particular on the sphere $S^n$.   Following work of Chen, Wei and Yan in dimension $n \geq 5$ \cite{CWY}, V\'etois and Wang construct an infinite family of positive solutions to the equation
$$
- \Delta u + ku = u^3
$$
on $S^4$ with $k$ constant $>2$ \cite{VW}.  Note that this corresponds to the biharmonic conformal map equation \eqref{cc} with $a = 3$ and $A=-1$.  We can therefore state the following consequence.

\begin{corollary}  \label{cor:sph} There is an infinite family of conformal metrics $\lambda^2g$ on the Euclidean sphere $(S^4, g)$ which are proper biharmonic.  In particular, there are proper biharmonic conformal submersions defined on the $4$-sphere.  Furthermore, each metric $h = \lambda^2g$ has scalar curvature $R_h =  6 -  \frac{6}{\lambda^2}$.  
\end{corollary} 

In fact the construction of V\'etois and Wang produces a family of positive solutions $u_{\epsilon}$ with the property that $||\nabla u_{\epsilon}||_{L^2(S^4)} \rightarrow \infty$ as $\epsilon \rightarrow 0$.  The solutions display spiking phenomena, with peaks that increase in height and number as $\epsilon \rightarrow 0$.  This translates into the same phenomena for the conformal metrics on $S^4$.  

Note that by the theorem of Jiang, necessarily, the conformally related metric $h=\lambda^2g$ must have positive sectional curvature somewhere.  On the other hand a theorem of Lohkamp guarantees that any manifold of dimension $n\geq 3$ supports a metric of strictly negative Ricci (and hence scalar) curvature \cite{Lo}, thus there is no constraint on the conformal factor implied by the expression for the scalar curvature $R_h$ in the above corollary.

\section{Biharmonicity of the M\"obius transformations} \label{sec:mobius}
The conformal transformations of the Euclidean sphere $(S^n, g_S)$ viewed as flat to flat conformal rescalings via stereographic projection form a group ${\rm Conf}(S^n)$ called the M\"obius group.  When we identify $S^n$ with $\r^n\cup \{\infty\}$ via sterographic projection, this group is generated by translations, dilations, the action of the orthogonal group and inversion in the unit sphere.  Specifically, any element of the M\"obius group can be written as:
\begin{equation} \label{Mob}
x\mapsto a + \frac{\alpha A(x-b)}{|x-b|^{\varepsilon}} \quad a,b \in \r^n, \ \alpha \in \r\setminus\{ 0\} , \ A \in O(n),\ \varepsilon \in \{ 0, 2\},
\end{equation}
for all $x \in \r^n$.  
 
When $n\geq 3$, Liouville's Theorem affirms that any local flat to flat conformal transformation is the restriction of a global conformal transformation (see, for example \cite{Ha} or \cite{BW}):
\begin{theorem} {\rm (Liouville Theorem)} For $n\ge 3$, and any open subset $\Omega\subseteq \mathbb{R}^n$, if $f:\Omega\rightarrow \mathbb{R}^n$ is a conformal transformation with respect to the standard Euclidean metric, then $f$ is the  restriction of a conformal transformation $\tilde{f}\in {\rm Conf}(S^n)$.  
\end{theorem}


Clearly, with respect to the Euclidean metrics on the domain and codomain, the conformal mapping \eqref{Mob} has conformal factor given by
$$
\lambda (x) = \frac{\alpha}{|x-b|^{\varepsilon}}
$$
whose Laplacian is given by $\Delta \lambda = \varepsilon (\varepsilon - 2)/|x-b|^{\varepsilon + 2}$, which therefore vanishes and so in particular the transformation is proper biharmonic whenever $\varepsilon = 2$.  When $\varepsilon = 0$, the transformation is a homothety with constant conformal factor, and therefore harmonic.  

Let us now consider the same transformation as a mapping from Euclidean space $\r^4$ with its canonical metric $dx^2$ into $\r^4$ with the spherical metric $g_S = 4dx^2/(1+|x|^2)^2$.  If $\phi : (M^n, g) \rightarrow (N^n, h)$ and $\psi : (N^n, h) \rightarrow (P^n,k)$ are two conformal mappings with conformal factors $\mu$ and $\nu$ respectively, then the composition is conformal with conformal factor given by $\lambda (x) = \nu (\phi (x))\mu (x)$ at each point $x \in M^n$.   Thus, viewed as a mapping $(\r^4, dx^2) \rightarrow (\r^4, 4dx^2/(1+|x|^2)^2)$, the transformation \eqref{Mob} has conformal factor
\begin{eqnarray*}
\lambda (x) &  = & \frac{\alpha}{|x-b|^{\varepsilon}} \times \frac{2}{1+\left| a + \frac{\alpha A(x-b)}{|x-b|^{\varepsilon}}\right|^2} \\
& = & \frac{2\alpha}{(1+|a|^2)|x-b|^{\varepsilon} + 2\alpha <a, A(x-b)> + \alpha^2|x-b|^{2-\varepsilon}}
\end{eqnarray*}
where $<\, \cdot  \, , \, \cdot \, >$ denotes the Euclidean inner product.  

In the case when $\varepsilon = 2$, this becomes
$$
\lambda (x) = \frac{2\alpha}{(1+|a|^2)\left| x-b + \frac{\alpha A^ta}{1+|a|^2}\right|^2 + \frac{\alpha^2}{1+|a|^2}}
$$
where $A^t$ denotes the transpose of the orthogonal matrix $A$.  On writing $\delta = \alpha / (1+|a|^2)$ and $e = b - \frac{\alpha A^ta}{1+|a|^2}$, this has the general form
\begin{equation} \label{gf}
\lambda (x) = \frac{2 \delta}{\delta^2+|x-e|^2}
\end{equation}
for a constant $\delta$ and constant vector $e\in \r^4$.  But this is precisely the form given by Corollary \ref{cor:AS4}, and in particular $\Delta \lambda = - 2 \lambda^3$, so that by \eqref{4D}, the M\"obius transform is proper biharmonic with respect to the flat Euclidean metric on the domain and the standard spherical metric on the codomain. 

In the case when $\varepsilon = 0$, then we have
$$
\lambda (x) = \frac{2/\alpha }{\frac{1}{\alpha^2} + \left| x - b + \frac{A^ta}{\alpha}\right|^2}\,,
$$
which, on now writing $\delta = 1/\alpha$ and $e = b - \frac{A^ta}{\alpha}$, has the same form \eqref{gf} 
and once more, the transformation is proper biharmonic.  The above calculations may be summarized by the following corollary.

\begin{corollary}  Any M\"obius transformation {\rm \eqref{Mob}} viewed as a map from Euclidean space $\r^4$ into either Euclidean space $\r^4$ or the standard sphere $S^4$ is biharmonic.  In the first case, it is proper biharmonic if and only if $\varepsilon \neq 0$; in the second case it is always proper biharmonic.   
\end{corollary}

Let us now consider the M\"obius transformations \eqref{Mob} defined with respect to the spherical metric $g_S = 4dx^2/(1+|x|^2)^2$ on the domain.  If we impose a metric of constant scalar curvarture on the codomain, then, since $a\neq 0$, \eqref{A-cf-scal} implies that the conformal factor $\lambda$ is constant and so necessarily the transformation is harmonic.  However,  once more, by the composition rule, the conformal factor is given by
$$
\lambda (x) = \frac{1+|x|^2}{2} \times \frac{\alpha}{|x-b|^{\varepsilon}}\,.
$$
which is never constant, whatever $\varepsilon\in \{ 0, 2\}$.  We therefore have the following consequence.   
\begin{corollary}
Viewed as a map from $\r^4$ endowed with the spherical metric $g_S = 4dx^2/(1+|x|^2)^2$ into $\r^4$ with the Euclidean metric, none of the M\"obius transformations {\rm \eqref{Mob}} are biharmonic.
\end{corollary}

A similar analysis can be made of the M\"obius transformations viewed as a map from $\r^4$ to $\r^4$ where both the domain and the codomain are endowed with the spherical metric.  Given the general form \eqref{gf} for the conformal factor of a M\"obius transformation from the flat metric to the spherical metric, by the multiplication rule for the conformal factor of a composition, we can suppose now that the conformal factor has the same form for both $\varepsilon = 0, 2$, namely 
$$
\lambda (x) = \frac{1+|x|^2}{2} \times \frac{2\delta}{\delta^2 + |x-e|^2} = \frac{\delta (1+|x|^2)}{\delta^2 + |x-e|^2}\,.
$$
If $\lambda = \lambda_0$ is constant, then
$$
\delta (1+|x|^2) = \lambda_0(\delta^2 + |x-e|^2).
$$
On comparing the coefficients of the different orders of $|x|$, in the case when $\varepsilon = 0$, this yields
$$
\lambda = 1 = \delta , \, e = 0 \quad \Rightarrow \quad \alpha = 1, b = A^ta \quad \Rightarrow \quad x \mapsto Ax\,.
$$
In the case when $\varepsilon = 2$, we have
$$
\lambda = \delta = 1,\, e = 0 \ \Rightarrow \ \alpha = 1+|a|^2, b = A^ta \ \Rightarrow \ 
x\mapsto a + \frac{(1+|a|^2)(Ax-a)}{|Ax-a|^2}\,.
$$
The former is an orthogonal transformation in the space orthogonal to the axis joining the north and south poles when we regard $S^4$ as the Euclidean sphere in $\r^5$.  The second is the same isometry followed by inversion about the point $a$, also an isometry of the sphere.  These two types of transformation generate all the isometries of the sphere.

\begin{corollary}  With respect to the metrics of constant postive curvature on $\r^4$, only the M\"obius transformations which are isometries are biharmonic, but these are also harmonic, so there is no proper biharmonic M\"obius transformation.
\end{corollary}

\end{document}